\documentclass[11pt,reqno]{amsart}
\usepackage{graphicx}
\usepackage{float}
\usepackage[caption = false]{subfig}
\usepackage{enumerate}
\usepackage{mathtools}
\usepackage{breqn}
\usepackage{array, geometry, graphicx}
\usepackage{amsmath,amsfonts,amssymb,amsthm}
\usepackage{graphicx}
\usepackage{amsmath,amscd,amsfonts,amssymb,enumerate, mathrsfs,wasysym,float,caption}
\usepackage{multirow}
\newtheorem{theorem}{Theorem}[section]
\newtheorem{corollary}[theorem]{Corollary}
\newtheorem{lemma}[theorem]{Lemma}

\theoremstyle{definition}

\theoremstyle{remark}

\numberwithin{equation}{section}
\DeclareMathOperator{\RE}{Re}

\begin{document}

\title{Radius Constants of Sigmoid Starlike Functions}
\thanks{$^*$Corresponding Author\\The first author is supported by The Council of Scientific and Industrial Research(CSIR). Ref.No.:08/133(0018)/2017-EMR-I.}
\author{Priyanka Goel}
\address{Department of Applied Mathematics, Delhi Technological University, Delhi--110042, India}
\email{priyanka.goel0707@gmail.com}
\author[S. S. Kumar]{S. Sivaprasad Kumar$^*$}
\address{Department of Applied Mathematics, Delhi Technological University, Delhi--110042, India}
\email{spkumar@dce.ac.in}

\subjclass[2020]{30C45, 30C80}

\keywords{Radius problem, Sigmoid function, Starlike functions}
\begin{abstract}
In the present investigation, we study the class of Sigmoid starlike functions, given by
$\mathcal{S}^*_{SG}=\{f\in\mathcal{A}: {zf'(z)}/{f(z)}\prec 2/(1+e^{-z})\}$
in context of estimating the sharp radius constants associated with several known subclasses of starlike functions. Further, graphical validation for the sharpness of results is also provided.
\end{abstract}

\maketitle

\section{Introduction}
Let $\mathcal{A}_n$ be the class of analytic functions defined on $\mathbb{D}:=\{z\in\mathbb{C}:|z|<1\}$, satisfying $f(0)=0,\;f'(0)=1$  and
of the form
\begin{equation*}
f(z)=z+a_{n+1}z^{n+1}+a_{n+2}z^{n+2}+\cdots,\quad n\in\mathbb{N}.
\end{equation*}
We denote by $\mathcal{A}:=\mathcal{A}_1$ and let $\mathcal{S}\subset\mathcal{A}$ be the class of univalent functions. For the functions $f$ and $F,$ we say that $f$ is subordinate to $F,$ written as $f\prec F,$ if it is possible to write $f(z)=F(\omega(z))$ for some Schwarz function $\omega.$ Let $\mathcal{S}^*$ and $\mathcal{C}$ denote respectively the class of starlike and convex functions. Note that if we consider $f\in\mathcal{A}_n,$ the class of starlike and convex functions are denoted by $\mathcal{S}^*_n$ and $\mathcal{C}_n$ respectively. For $n\in\mathbb{N},$ we now define the Carath\'{e}odory class $\mathcal{P}_n,$ containing functions of the form $p(z)=1+c_nz^n+c_{n+1}z^{n+1}+\cdots$ with $\RE(p(z))>0$ on $\mathbb{D}.$ Using subordination, Ma and Minda~\cite{maminda} defined a general subclass of starlike functions, given by $\mathcal{S}^*(\phi):=\{f\in\mathcal{A}:zf'(z)/f(z)\prec\phi(z)\}.$ For different choices of $\phi,$ authors have defined several subclasses of $\mathcal{S}^*$ and examined these classes for different geometric properties. These classes will be described in the text wherever needed. In 2020, we introduced the class of Sigmoid starlike functions by taking $\phi(z)=2/(1+e^{-z}),$ the Modified Sigmoid function and denote it by $\mathcal{S}^*_{SG}$~\cite{first}. The image of $\mathbb{D}$ under the Modified Sigmoid function is denoted by $\Delta_{SG}:=\{w\in\mathbb{C}:|\log(w/(2-w))|<1\}$. In~\cite{first}, we present some basic geometry of this function, prove several inclusion relationships, obtain some coefficient bounds, and study mainly first-order differential subordination results. Later in~\cite{third}, we proved various second and third order differential subordination results for Sigmoid starlike functions. Soon the class gained popularity and attracted many authors to study further in context of various aspects such as coefficient problems and convolution results (see~\cite{ghaff1,ghaff2}). In continuation of these works, we now investigate $\mathcal{S}^*_{SG}$ for radius problems. We present in this paper, radius estimates for $\mathcal{S}^*_{SG}$ in conjunction with a bunch of other subclasses of starlike functions. Further, we consider certain families of analytic functions, which are characterized by the ratio of its functions with a specific function $g$ and obtain $\mathcal{S}^*_{SG}$- radius for these families. We extensively use the following lemma in order to prove our main results:
\begin{lemma}\cite{first}\label{main}
	Let $2/(1+e)<a<2e/(1+e)$. If  $$r_a=\dfrac{e-1}{e+1}-|a-1|,$$ then
	\begin{equation}\label{largestdisk}
		\{w\in\mathbb{C}:|w-a|<r_a\}\subset\Delta_{SG}.
	\end{equation}
\end{lemma}
\section{Main Results}
To begin with, we consider the class $\mathcal{S}^*_n[A,B]\;(-1\leq B<A\leq1),$ introduced by Janowski~\cite{janowski}, as below.
\begin{equation*}
\mathcal{S}^*_n[A,B]:=\left\{f\in\mathcal{A}_n:\frac{zf'(z)}{f(z)}\prec\frac{1+Az}{1+Bz}\right\}.
\end{equation*}
Corresponding to the above class, we consider $\mathcal{P}_n[A,B]$ and a few results related to this class, given as follows:
\begin{equation*}
\mathcal{P}^*_n[A,B]:=\left\{f\in\mathcal{P}_n:p(z)\prec\frac{1+Az}{1+Bz}\right\}.
\end{equation*}
\begin{lemma}~\cite[Lemma~2.1]{Plemma}\label{Plemma}
\begin{itemize}
\item[$(i)$] 	If $p\in\mathcal{P}_n[A,B]$, then for $|z|=r$, $$\left|p(z)-\dfrac{1-ABr^{2n}}{1-B^2r^{2n}}\right|\leq \dfrac{(A-B)r^n}{1-B^2r^{2n}}.$$
\item[$(ii)$]
	In particular, if $p\in\mathcal{P}_n(\alpha):=\mathcal{P}[1-2\alpha, -1]$, then
	$$\left|p(z)-\dfrac{1+(1-2\alpha)r^{2n}}{1-r^{2n}}\right|\leq \dfrac{2(1-\alpha)r^n}{1-r^{2n}}.$$
\end{itemize}
\end{lemma}
\begin{lemma}~\cite[Lemma~2]{lemma(b)}\label{lemmab}
If $p\in\mathcal{P}_n(\alpha),$ then for $|z|=r,$
\begin{equation*}
\left|\frac{zp'(z)}{p(z)}\right|\leq \frac{2nr^n(1-\alpha)}{(1-r^n)(1+(1-2\alpha)r^n)}.
\end{equation*}
\end{lemma}
\begin{theorem}\label{Sn_ab_1}
		The sharp $\mathcal{S}^*_{SG,n}$-radius of the class $\mathcal{S}^*_n[A,B]$ is given by
		\begin{itemize}
			\item[$(i)$]
			
			$R_{\mathcal{S}^*_{SG,n}}(\mathcal{S}^*_n[A,B]) = \min \left\{1,\; \left(\frac{e-1}{A(1+e)-2B}\right)^{\frac{1}{n}} \right\},$ when $0\leq B < A \leq 1.$
			
			\item[$(ii)$]
			
		$R_{\mathcal{S}^*_{SG,n}}(\mathcal{S}^*_n[A,B]) = \min \left\{1,\; \left(\frac{e-1}{A(1+e)-2Be}\right)^{\frac{1}{n}} \right\},$ when $-1 \leq B < A \leq 1$ with $B\leq 0.$
			
		\end{itemize}
		In particular, for the class $\mathcal{S}^*$, we have $R_{\mathcal{S}^*_{SG}}(\mathcal{S}^*)=(e-1)/(3e+1)$.
	\end{theorem}
\begin{proof}
		Let $f \in \mathcal{S}^*_n[A,B]$. Using Lemma~\ref{Plemma}, we have
		\begin{equation}\label{snab}
		\left| \frac{zf'(z)}{f(z)} -\frac{1-ABr^{2n}}{1-B^2r^{2n}} \right| \leq \frac{(A-B)r^n}{1-B^2r^{2n}}.
		\end{equation}
		{\bf(i)} If $0\leq B< A\leq1$, then
		
		$$a := \frac{1-ABr^{2n}}{1-B^2r^{2n}}\leq 1.$$
		Further, we observe that $f\in\mathcal{S}^*_{SG,n}$ if the disk~\eqref{snab} is contained in the disk~\eqref{largestdisk}. Thus by using Lemma~\ref{main}, it suffices to show that
		$$\frac{(A-B)r^n}{1-B^2r^{2n}}\leq \frac{1-ABr^{2n}}{1-B^2r^{2n}}-\frac{2}{1+e},$$
		which upon simplification, yields
		$$r\leq \left(\frac{e-1}{A(1+e)-2B}\right)^{\frac{1}{n}}.$$
{\bf(ii)}	If $-1 \leq B < 0 < A \leq 1$, then
		$$a := \frac{1-ABr^{2n}}{1-B^2r^{2n}}\geq 1.$$
			Again by Lemma~\ref{main}, we see that $f \in \mathcal{S}^*_{SG,n}$ if
		$$\frac{(A-B)r^n}{1-B^2r^{2n}}\leq \frac{2e}{1+e}-\frac{1-ABr^{2n}}{1-B^2r^{2n}},$$
		which simplifies to
		$$r\leq \biggl(\frac{e-1}{A(1+e)-2Be}\biggl)^{1/n}.$$ The result follows with sharpness due to the function $f_{A,B}(z),$ given by
		\begin{equation*}
		f_{A,B}(z) =
		\left\{
		\begin{array}{ll}
		z(1+Bz^n)^{\frac{A-B}{nB}}; & B\neq0, \\
		z\exp\left(\frac{Az^n}{n}\right);    &  B=0.
		\end{array}	
		\right.
		\end{equation*}
\end{proof}
\begin{corollary}
  The sharp $\mathcal{S}^*_{SG}$-radius for $\mathcal{S}^*(\alpha)$ is $(e-1)/(1+3e-2\alpha(1+e)),\;0\leq\alpha<1.$ The bound is sharp for $k_{\alpha}(z)=z/(1-z)^{2(1-\alpha)}.$
\end{corollary}
\begin{corollary}
  The sharp $\mathcal{S}^*_{SG}$-radius for $\mathcal{S}^*$ is $(e-1)/(1+3e).$ The bound is sharp for $k(z)=z/(1-z)^{2}.$
\end{corollary}
Before we proceed to our next result, we need to recall the following classes:\\
For $0\leq\alpha<1,$ the class $\mathcal{BS}^*(\alpha):=\{f\in\mathcal{A}:zf'(z)/f(z)\prec 1+z/(1-\alpha z^2)\},$ defined by Kargar et al.~\cite{kargar}. In~\cite{kanika}, Khatter et al. generalized
$\mathcal{S}^*_L:=\mathcal{S}^*(\sqrt{1+z})$ and $\mathcal{S}^*_e:=\mathcal{S}^*(e^z)$ to $\mathcal{S}^*_L(\alpha):=\mathcal{S}^*(\alpha+(1-\alpha)\sqrt{1+z})$ and $\mathcal{S}^*_{\alpha,e}:=\mathcal{S}^*(\alpha+(1-\alpha)e^z)$ respectively, for $\alpha\in[0,1)$.

\begin{theorem}
The radius estimates of Sigmoid starlikeness, for the classes $\mathcal{BS}^*(\alpha),\;\mathcal{S}^*_L(\alpha)$ and $\mathcal{S}^*_{\alpha,e}$ are given by
\begin{itemize}
 \item[$(i)$] $R_{\mathcal{S}^*_{SG}}(\mathcal{BS}^*(\alpha))=r_{\mathcal{BS}}(\alpha):=2(e-1)/((1+e)+\sqrt{(1+e)^2+4\alpha(e-1)^2})$, where $\alpha\in[0,1).$
 \item[$(ii)$] $R_{\mathcal{S}^*_{SG}}(\mathcal{S}^*_L(\alpha))=r_L(\alpha):=((e-1)(3+e-2\alpha(1+e)))/((1-\alpha)^2(1+e)^2),\;\text{where}\; \alpha\in[0,(3+e)/2(1+e)).$ In particular, $R_{\mathcal{S}^*_{SG}}(\mathcal{S}^*_L)=((e-1)(3+e))/(1+e)^2.$
 \item[$(iii)$] $R_{\mathcal{S}^*_{SG}}(\mathcal{S}^*_{\alpha,e})=r_e(\alpha):=\log{(2e-\alpha(1+e))/(1+e)(1-\alpha)},$ where $\alpha\in[0,(e(1+e)-2e)/((1+e)(e-1))).$ In particular, $R_{\mathcal{S}^*_{SG}}(\mathcal{S}^*_e)=\log(2e/(1+e)).$
 \end{itemize}
 All estimates are sharp.
\end{theorem}
\begin{proof}
\begin{itemize}
\item[(i)]  Let $f\in\mathcal{BS}^*(\alpha).$ Then $zf'(z)/f(z)\prec 1+z/(1-\alpha z^2)$ and thus
    \begin{equation*}
      \left|\dfrac{zf'(z)}{f(z)}-1\right|\leq\left|\frac{z}{1-\alpha z^2}\right|\leq \left|\frac{r}{1-\alpha r^2}\right|\quad \text{on }|z|=r.
    \end{equation*}
    Using Lemma~\ref{main}, it can be said that the above disk lies in $\Delta_{SG}$ if $r/(1-\alpha r^2)\leq (e-1)/(e+1).$ This further implies $r\leq r_{\mathcal{BS}}(\alpha).$ Sharpness holds for the function
    \begin{equation*}
      f_{\mathcal{BS}}(z)=\begin{cases}
        z\left(\frac{1+\sqrt{\alpha z}}{1-\sqrt{\alpha}z}\right)^{1/(2\sqrt{\alpha})},& \alpha\in(0,1)\\
        ze^z, & \alpha=0.
      \end{cases}
    \end{equation*}
    It can be verified with the following graph that $zf'_{\mathcal{BS}}(z)/f_{\mathcal{BS}}(z)$ touches the boundary of $\Delta_{SG}$ at the points $\pm2(e-1)/((1+e)+\sqrt{(1+e)^2+4\alpha(e-1)^2}).$ Note that the domain $\Omega_{BS}$ denotes the image of $\mathbb{D}$ mapped by the function $1+z/(1-\alpha z^2)$\\ \ \\
\begin{figure}[H]
\begin{tabular}{cl}

         \begin{tabular}{c}\label{fig1}
          \includegraphics[height=5cm, width=5cm]{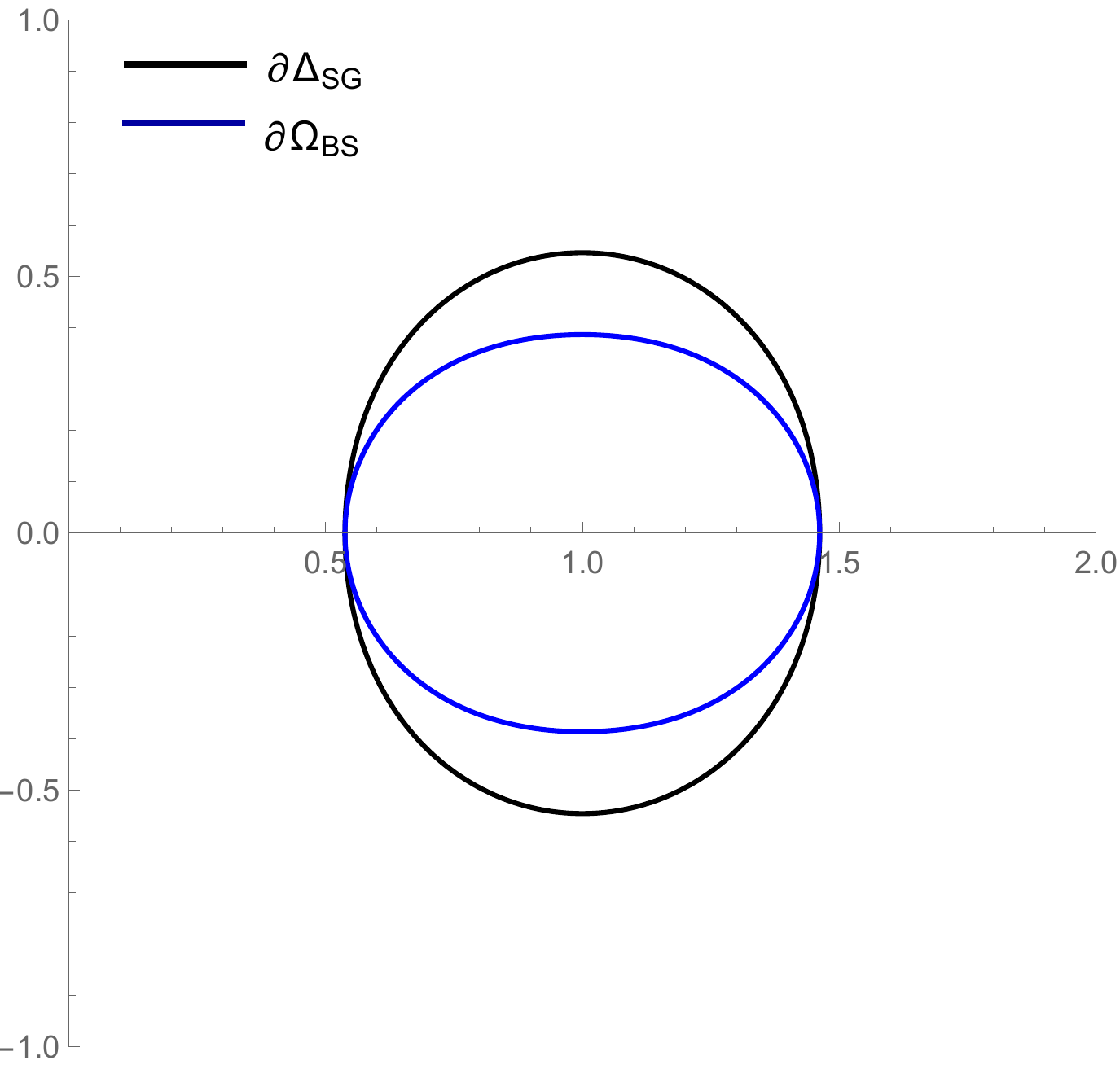}\\
          Sharpness for $\alpha=0.5$

           \end{tabular}

           & \begin{tabular}{l}
             \parbox{0.4\linewidth}{
            \includegraphics[height=5cm, width=5cm]{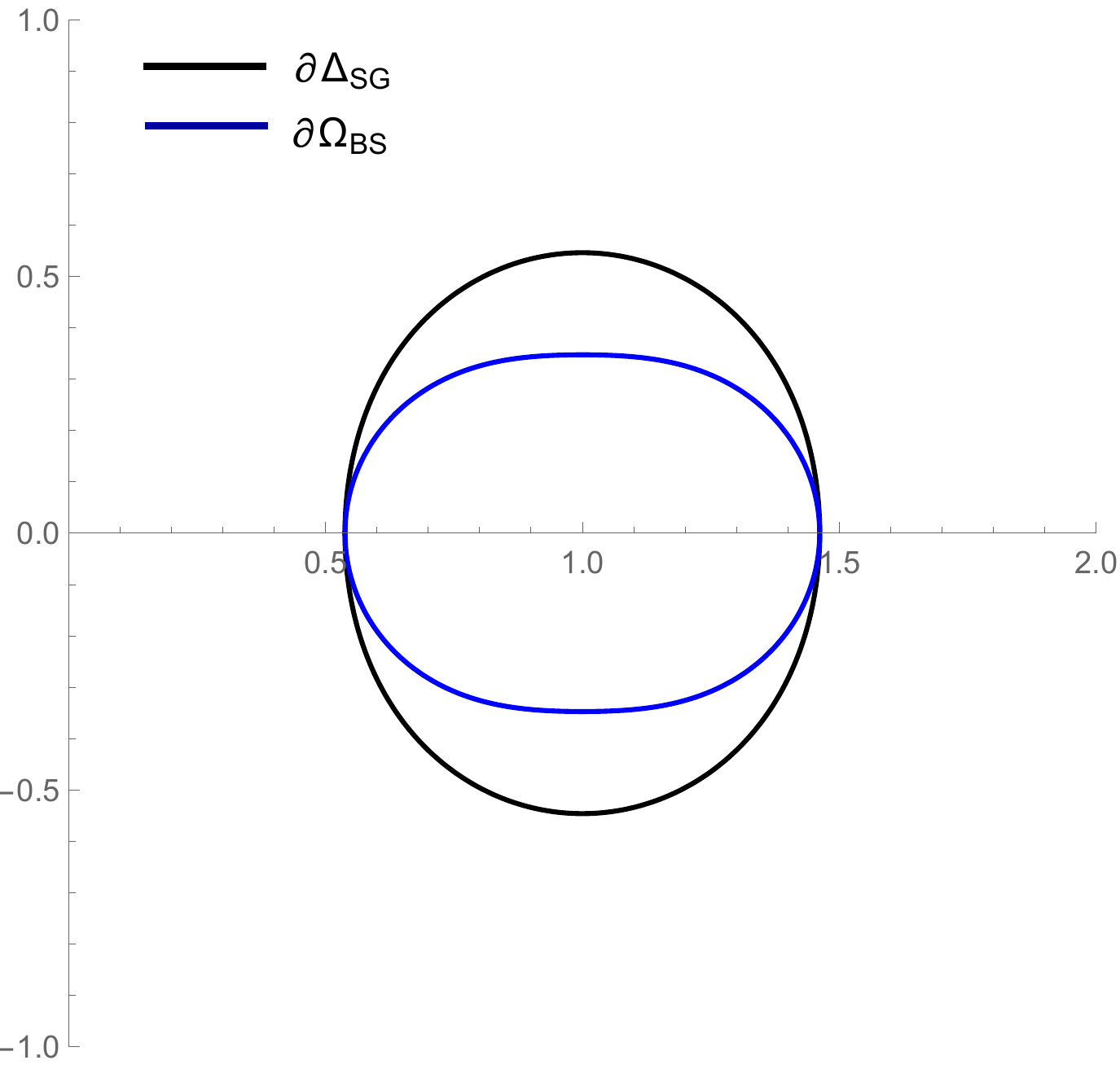}\\
               Sharpness for $\alpha=0.9$ }
         \end{tabular}  \\
\end{tabular}
\caption{}
\end{figure}
\item[(ii)]  Let $f\in\mathcal{S}^*_L(\alpha),$  then $zf'(z)/f(z)\prec \alpha+(1-\alpha)\sqrt{1+z}$ and therefore on $|z|=r$
        \begin{equation*}
          \left|\frac{zf'(z)}{f(z)}-1\right|\leq |(1-\alpha)(1-\sqrt{1+z})|\leq (1-\alpha)(1-\sqrt{1-r}).
        \end{equation*}
        By Lemma~\ref{main}, it is clear that for the above disk to lie in $\Delta_{SG},$ we need $(1-\alpha)(1-\sqrt{1-r})\leq (e-1)/(e+1),$ which upon simplification yields $r\leq ((e-1)(3+e-2\alpha(1+e)))/((1-\alpha)^2(1+e)^2).$ Note that for the function
        \begin{equation*}
          f_L(z)=z+(1-\alpha)z^2+\frac{1}{16}(1-\alpha)(1-2\alpha)z^3+\cdots,
        \end{equation*}
        the result is sharp. The sharpness of this result can be verified by the following graph, where $\Omega_L$ denotes the image of $\mathbb{D}$ mapped by $\alpha+(1-\alpha)\sqrt{1+z}$.\\
        \begin{figure}[H]
   \begin{tabular}{cl}

         \begin{tabular}{c}\label{fig1}
          \includegraphics[height=5cm, width=5cm]{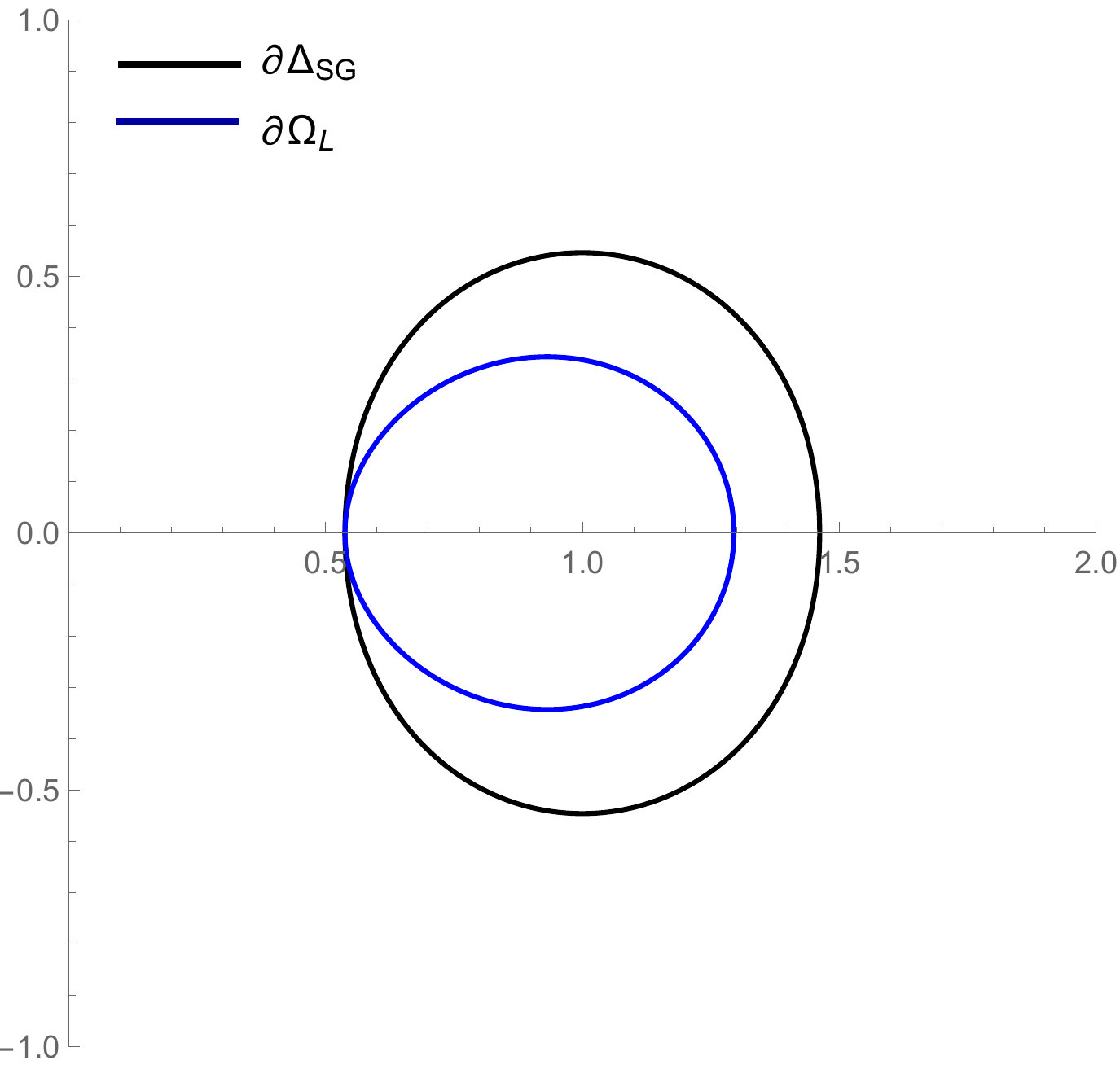}\\
          Sharpness for $\alpha=0.1$

           \end{tabular}

           & \begin{tabular}{l}
             \parbox{0.4\linewidth}{
            \includegraphics[height=5cm, width=5cm]{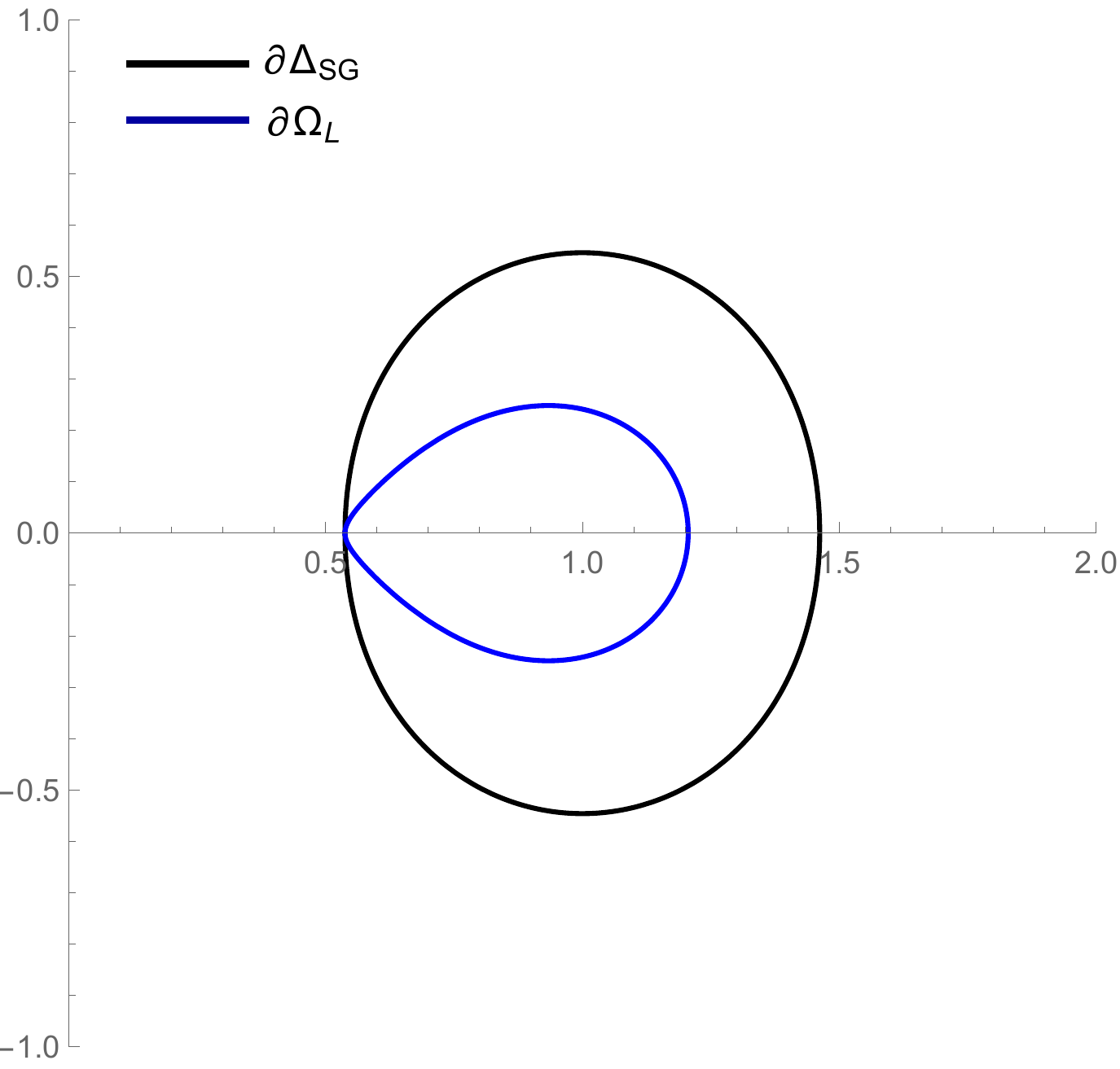}\\
               Sharpness for $\alpha=0.5$

    }
         \end{tabular}  \\
\end{tabular}
\caption{}
\end{figure}
\item[(iii)]    Let $f\in\mathcal{S}^*_{\alpha,e},$ then $zf'(z)/f(z)\prec \alpha+(1-\alpha)e^z.$ So on $|z|=r$
        \begin{equation*}
          \left|\frac{zf'(z)}{f(z)}-1\right|\leq (1-\alpha)|e^z-1|\leq (1-\alpha)(e^r-1).
        \end{equation*}
        By Lemma~\ref{main}, $f\in\mathcal{S}^*_{SG}$ if $(1-\alpha)(e^r-1)\leq (e-1)/(e+1),$ which is equivalent to $r\leq r_e(\alpha).$ The result is sharp for the function
        \begin{equation*}
          f_e(z)=z+(1-\alpha)z^2+\frac{1}{4}(1-\alpha)(3-2\alpha)z^3+\cdots
        \end{equation*}
        and is validated by the following graph. The image of $\mathbb{D}$ mapped by $\alpha+(1-\alpha)e^z$ is denoted by $\Omega_e$.\\
        \begin{figure}[H]
     \begin{tabular}{cl}

         \begin{tabular}{c}\label{fig1}
          \includegraphics[height=5cm, width=5cm]{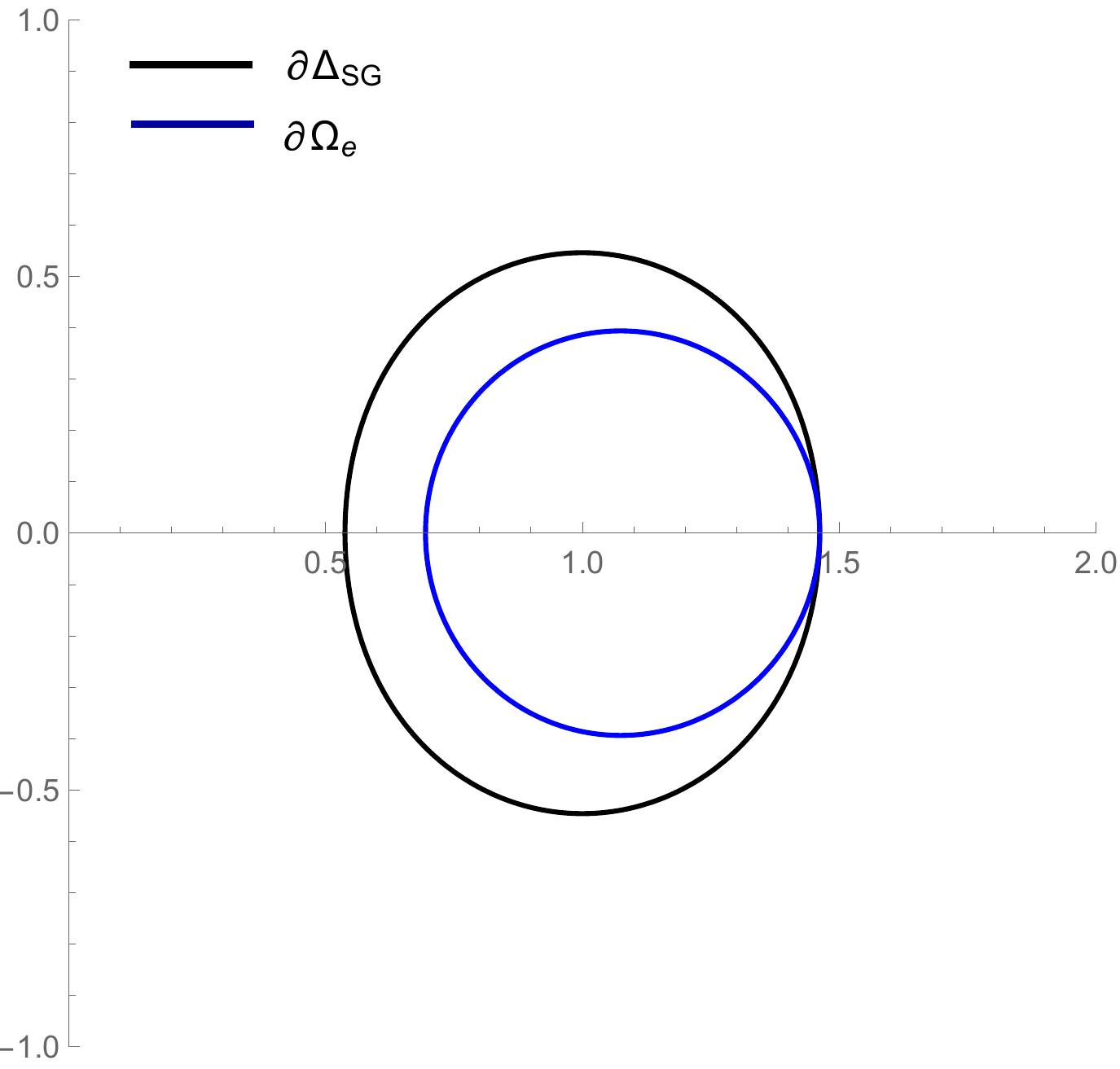}\\
          Sharpness for $\alpha=0.1$

           \end{tabular}

           & \begin{tabular}{l}
             \parbox{0.4\linewidth}{
            \includegraphics[height=5cm, width=5cm]{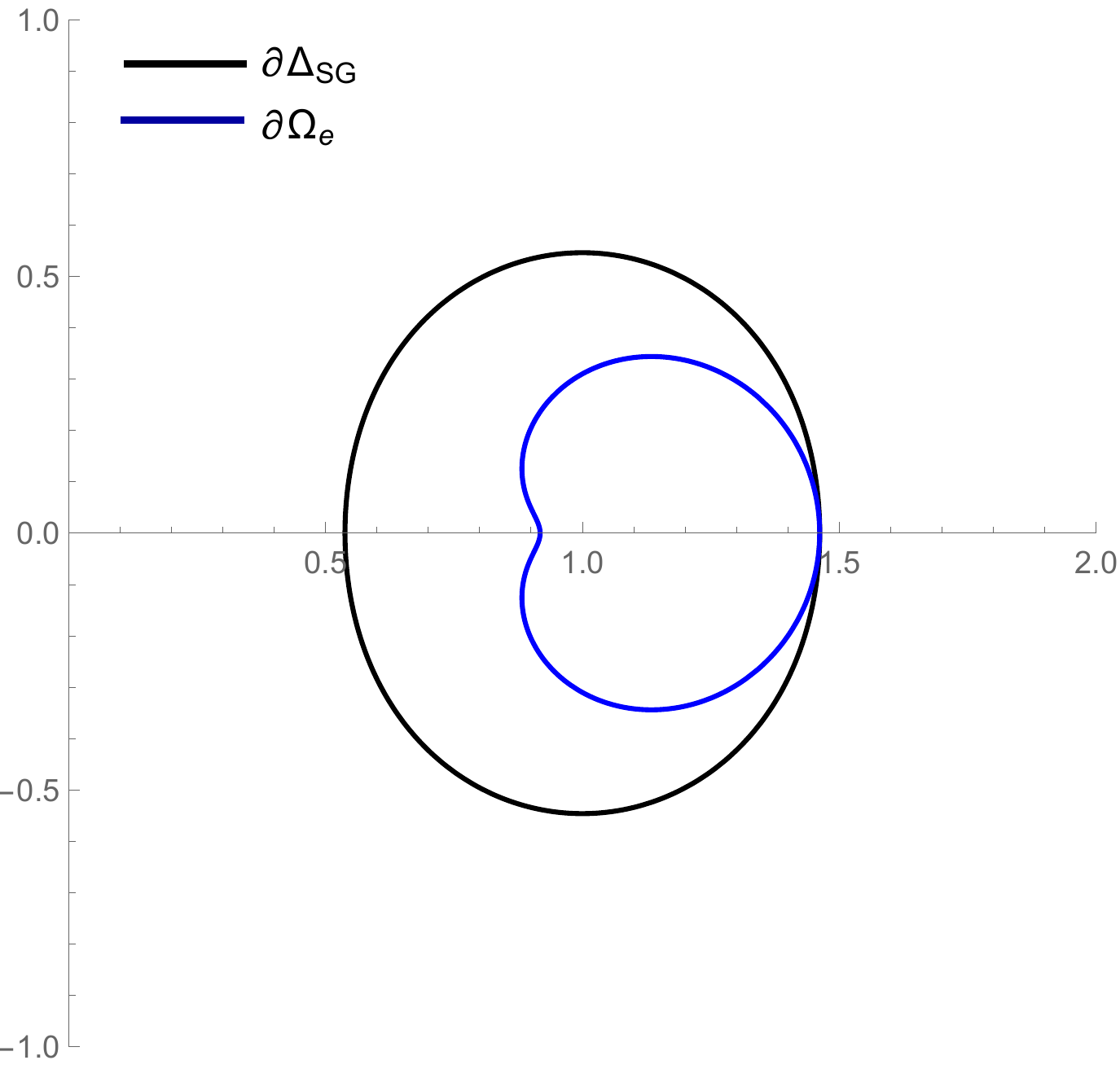}\\
               Sharpness for $\alpha=0.9$
               }
         \end{tabular} \\

\end{tabular}
\caption{}
\end{figure}
\end{itemize}
\end{proof}
Before we proceed further, let us recall the following classes: In~\cite{mendiratta2}, Mendiratta et al. considered the class of starlike functions associated with right lemniscate of Bernoulli, denoted by $\mathcal{S}^*_{RL}=\mathcal{S}^*(\phi),$ where $\phi$ is given by
$$\phi(z)=\sqrt{2}-(\sqrt{2}-1)\sqrt{\frac{1-z}{1+2(\sqrt{2}-1)z}}.$$
The class of cardioid starlike functions, denoted by $\mathcal{S}^*_C:=\mathcal{S}^*(1+4z/3+2z^2/3),$ defined by Sharma et al.~\cite{cardiod}. For $k=\sqrt{2}+1,$ Kumar and Ravichandran~\cite{rational} introduced $\mathcal{S}^*_{R}$ by taking $\phi$ as $1+z(k+z)/k(k-z).$
\begin{theorem}
  The sharp $\mathcal{S}^*_{SG}$-radius for the classes $\mathcal{S}^*_{RL},\;\mathcal{S}^*_C$ and $\mathcal{S}^*_R$ is given by:
  \begin{itemize}
    \item [$(i)$] $R_{\mathcal{S}^*_{SG}}(\mathcal{S}^*_{RL})=:r_{RL}=\left(4 \sqrt{2}-7 e-5\right) (e-1)/(32 \sqrt{2}-7 e^2+6 e \left(4 \sqrt{2}-5\right)-47)\approx 0.738309.$
    \item [$(ii)$]$R_{\mathcal{S}^*_{SG}}(\mathcal{S}^*_C)=:r_C=-1+\sqrt{(-1+5e)/(2+2e)}\approx 0.301221.$
    \item [$(iii)$]$R_{\mathcal{S}^*_{SG}}(\mathcal{S}^*_{R})=:r_R=(\sqrt{\left(2 \sqrt{2}+3\right) \left(2 e^2-1\right)}-\left(\sqrt{2}+1\right) e)/(1+e)\approx 0.645131.$
  \end{itemize}
\end{theorem}
\begin{proof}
\begin{itemize}
\item[(i)] Let $f\in\mathcal{S}^*_{RL}.$ Then
$$\frac{zf'(z)}{f(z)}\prec \sqrt{2}-(\sqrt{2}-1)\sqrt{\frac{1-z}{1+2(\sqrt{2}-1)z}}.$$
Thus on $|z|=r,$ we have
$$\left|\frac{zf'(z)}{f(z)}-1\right|\leq 1-\left(\sqrt{2}-(\sqrt{2}-1)\sqrt{\frac{1+r}{1-2(\sqrt{2}-1)r}}\right).$$
By Lemma~\ref{main}, $f$ is in $\mathcal{S}^*_{SG}$ if
$$1-\left(\sqrt{2}-(\sqrt{2}-1)\sqrt{\frac{1+r}{1-2(\sqrt{2}-1)r}}\right)\leq\frac{e-1}{e+1},$$
which is equivalent to $r\leq r_{RL}.$ This result is sharp for the following function
{\small$$f_{RL}(z)=z\left(\frac{\sqrt{1-z}+\sqrt{1+2(\sqrt{2}-1)z}}{2}\right)^{2\sqrt{2}-2}\exp{\left(\sqrt{2(\sqrt{2}-1)}\tan^{-1}\Psi(z)\right)},$$}
where
$$\Psi(z)=\frac{\sqrt{2(\sqrt{2}-1)}\left(\sqrt{2(\sqrt{2}-1)z+1}-\sqrt{1-z}\right)}{2(\sqrt{2}-1)\sqrt{1-z}+\sqrt{2(\sqrt{2}-1)z+1}}.$$
The sharpness of this bound can be verified from Figure~\ref{threefig}(i).
\item[(ii)] Let $f\in\mathcal{S}^*_C,$ then we have $zf'(z)/f(z)\prec 1+4z/3+2z^2/3.$ Therefore on $|z|=r,$ we get
  \begin{equation*}
    \left|\frac{zf'(z)}{f(z)}-1\right|\leq \frac{2(r^2+2r)}{3},
  \end{equation*}
  which if not exceeds $(e-1)/(e+1),$ implies that $f$ lies in $\mathcal{S}^*_{SG},$ by Lemma~\ref{main}. Solving this, we get $r\leq r_C.$ In order to verify the sharpness of this result, we consider the following function.
  \begin{equation*}
    f_C(z)=z\exp{\left(\frac{4z}{3}+\frac{z^2}{3}\right)}.
  \end{equation*}
  Clearly $f_C\in \mathcal{S}^*_C$ and moreover $zf'_C(z)/f_C(z)$ touches the boundary of $\Delta_{SG}$ at the point $z_0=-1+\sqrt{(-1+5e/(2+2e)},$ as shown in Figure~\ref{threefig}(ii).
\item[(iii)]  Let $f\in\mathcal{S}^*_R,$ then
  \begin{equation*}
    \frac{zf'(z)}{f(z)}\prec 1+\frac{z(k+z)}{k(k-z)},
  \end{equation*}
  where $k=\sqrt{2}+1.$ Thus on $|z|=r,$
  \begin{equation*}
    \left|\frac{zf'(z)}{f(z)}-1\right|\leq \frac{r(k+r)}{k(k-r)}.
  \end{equation*}
  In view of Lemma~\ref{main}, $f\in\mathcal{S}^*_{SG}$ if $r(k+r)/k(k-r)\leq (e-1)/(e+1).$ Solving this inequality, we obtain $r\leq r_R$. The equality of the radius estimate holds for the function
  $$f_R(z)=\dfrac{k^2 z}{(k-z)^2}e^{-z/k},\quad k=\sqrt{2}+1.$$
  Figure~\ref{threefig}(iii) verifies the sharpness of the result. Note that $\Omega_{RL},\;\Omega_{c}$ and $\Omega_R$ denote the image of $\mathbb{D}$ mapped by $zf'(z)/f(z)$ for $f_{RL},\;f_C$ and $f_R$ repectively.\\
\begin{figure}[!htb]
\minipage{0.32\textwidth}
  \includegraphics[width=\linewidth]{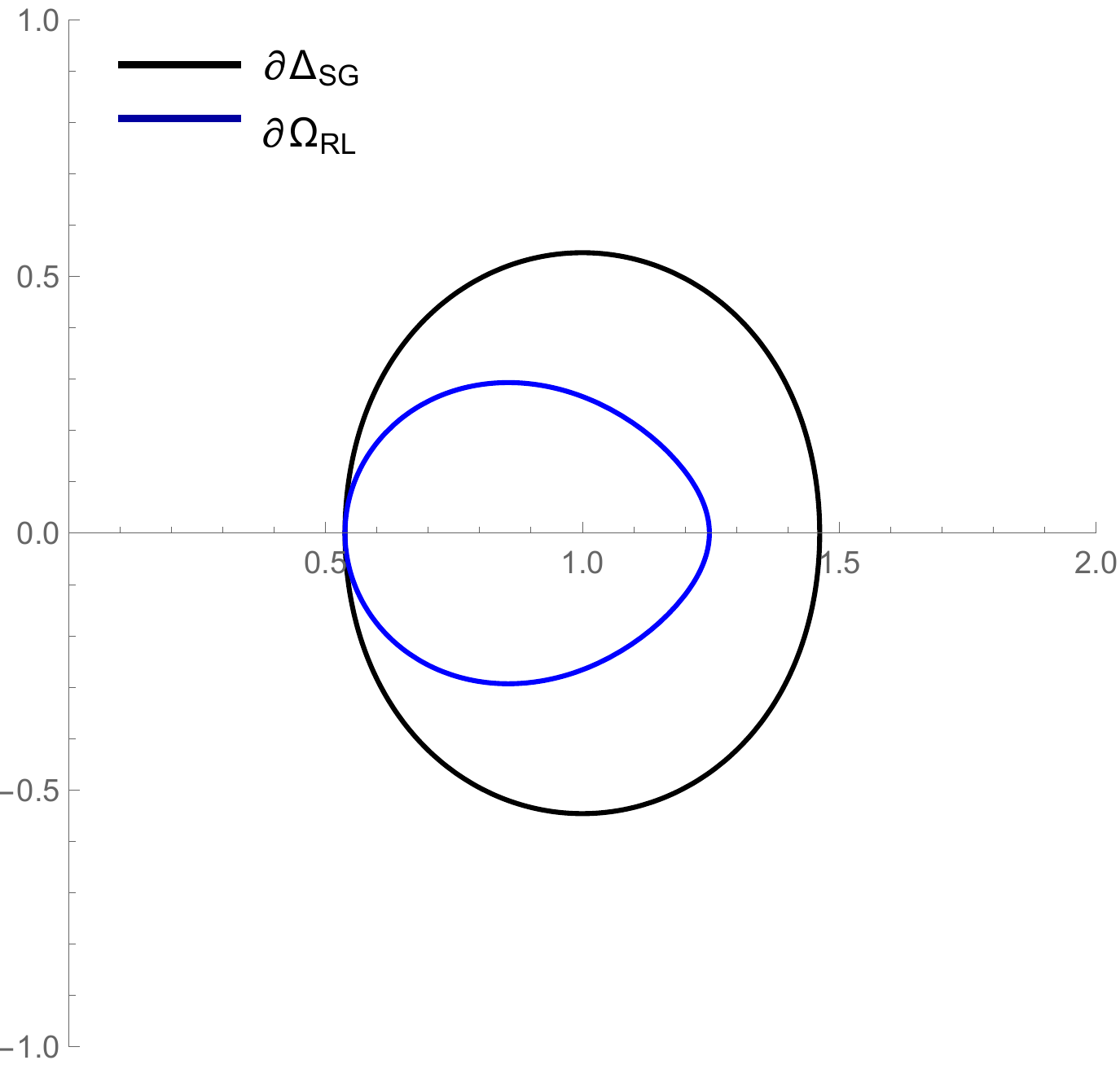}
  \captionsetup{labelformat=empty}
\endminipage\hfill
\minipage{0.32\textwidth}
  \includegraphics[width=\linewidth]{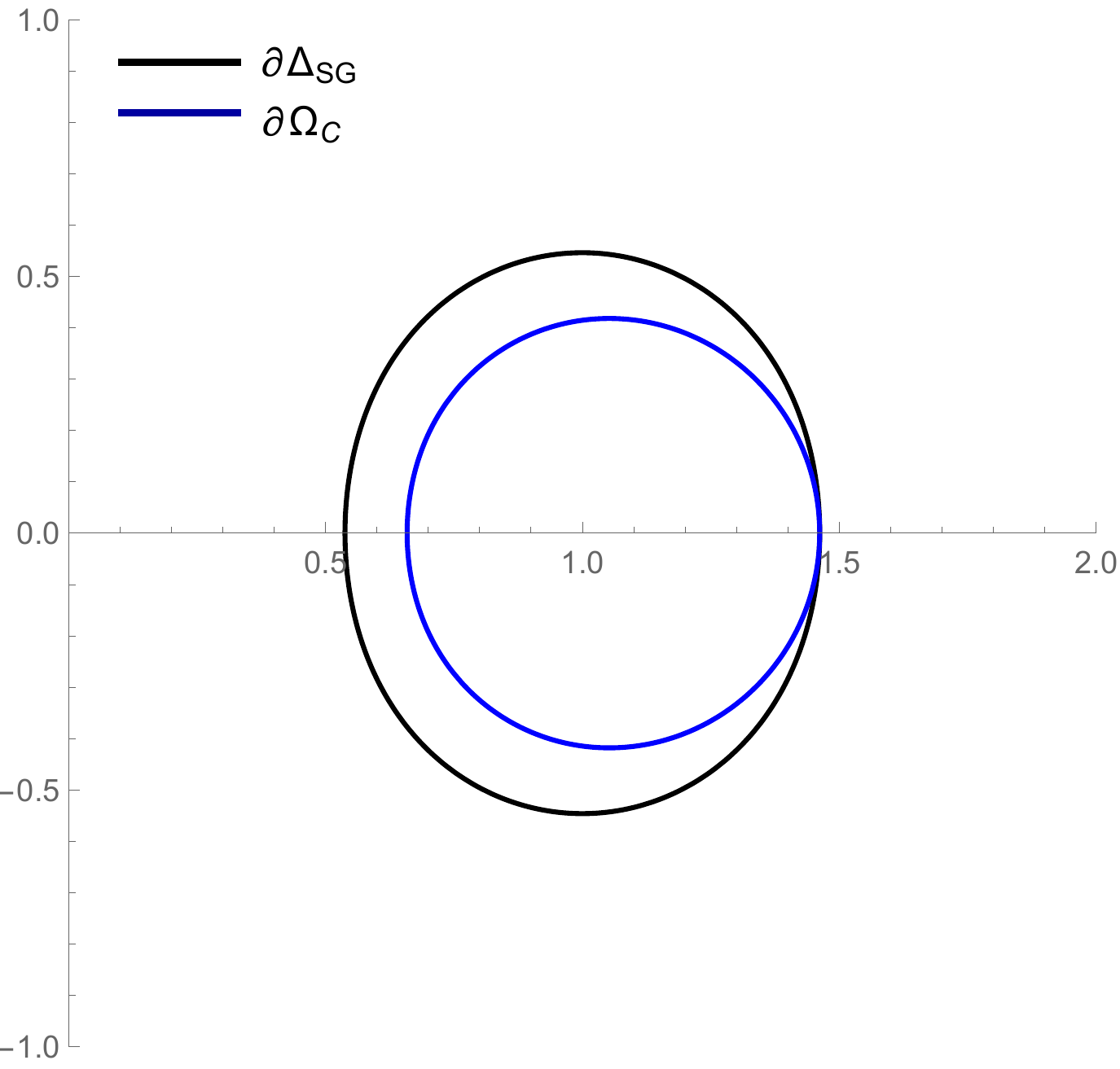}
  \captionsetup{labelformat=empty}
\endminipage\hfill
\minipage{0.32\textwidth}%
  \includegraphics[width=\linewidth]{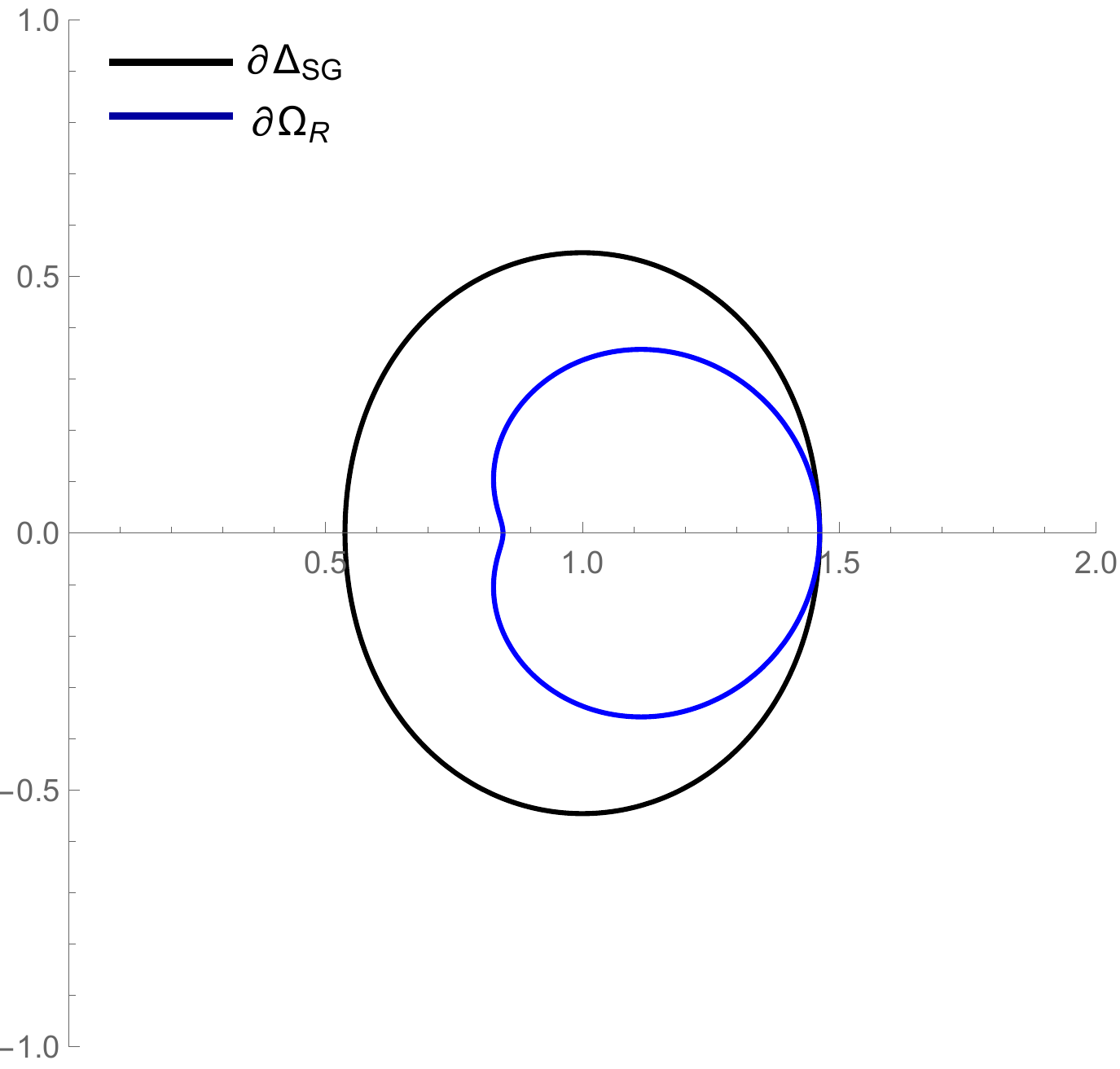}
  \captionsetup{labelformat=empty}
\endminipage
\caption{}\label{threefig}
\end{figure}
\end{itemize}
\end{proof}
Let us recall the following classes in order to obtain our next result. By taking $\phi(z)=z+\sqrt{1+z^2}$, Sharma et al.~\cite{crescent} introduced $\mathcal{S}^*_{\leftmoon}.$ Similarly, Kumar and Gangania~\cite{kamal} intoduced $\mathcal{S}^*_{\wp}$ by taking $\phi$ as $1+ze^z,$ the cardioid function.
\begin{theorem}
 The sharp $\mathcal{S}^*_{SG}-$radii for the classes $\mathcal{S}^*_{\leftmoon}$ and $\mathcal{S}^*_{\wp}$  is given by
\begin{itemize}
 \item[$(i)$] $
   R_{\mathcal{S}^*_{SG}}(\mathcal{S}^*_{\leftmoon})=\frac{-1-2 e+3 e^2}{4 e+4 e^2}\approx 0.389089.$
 \item[$(ii)$] $ R_{\mathcal{S}^*_{SG}}(\mathcal{S}^*_{\wp})=r_{\wp}\approx 0.331672,$
 where $r_{\wp}$ is the smallest positive root of the equation $(e+1)re^r=e-1.$
\end{itemize}
\end{theorem}
\begin{proof}
\begin{itemize}
\item[$(i)$] Let $f\in\mathcal{S}^*_{\leftmoon},$ then $zf'(z)/f(z)\prec z+\sqrt{1+z^2}.$ Therefore on $|z|=r,$
\begin{equation*}
  \left|\frac{zf'(z)}{f(z)}-1\right|=|z+\sqrt{1+z^2}-1|\leq r+\sqrt{1+r^2}-1.
\end{equation*}
Now by using Lemma~\ref{main}, the above disk lies inside the domain $\Delta_{SG}$ if $r+\sqrt{1+r^2}-1\leq(e-1)/(e+1).$ Solving this equation, we obtain the desired bound of $r.$ The result is sharp for the function
\begin{equation*}
  f_{\leftmoon}(z)=z\exp{\left(\int_0^z\dfrac{t+\sqrt{1+t^2}-1}{t}dt\right)}=z+z^2+\dfrac{3z^3}{4}+\dfrac{5z^4}{12}+\dfrac{z^5}{6}+\cdots.
\end{equation*}
\item[$(ii)$]Let $f\in\mathcal{S}^*_{\wp},$ then it is clear that $zf'(z)/f(z)\prec 1+ze^z.$ So we have
$$\left|\dfrac{zf'(z)}{f(z)}-1\right|=|ze^z|\leq re^r\quad\text{on}\;|z|=r.$$
By using Lemma~\ref{main}, we can say that $f\in\mathcal{S}^*_{SG}$ if $re^r\leq (e-1)/(e+1),$ which is equivalent to $r\leq r_{\wp}.$ The sharpness of the result can be verified by the function $f_{\wp}(z)=ze^{e^z-1}.$ The following graph depicts the sharpness of both the estimates. Note that the image of $\mathbb{D}$ mapped by $z+\sqrt{1+z^2}$ and $1+ze^z$ are respectively denoted by $\Omega_{\leftmoon}$ and $\Omega_{\wp}.$
\end{itemize}
\begin{figure}[H]
 \begin{tabular}{cl}

         \begin{tabular}{c}\label{fig1}
          \includegraphics[height=5cm, width=5cm]{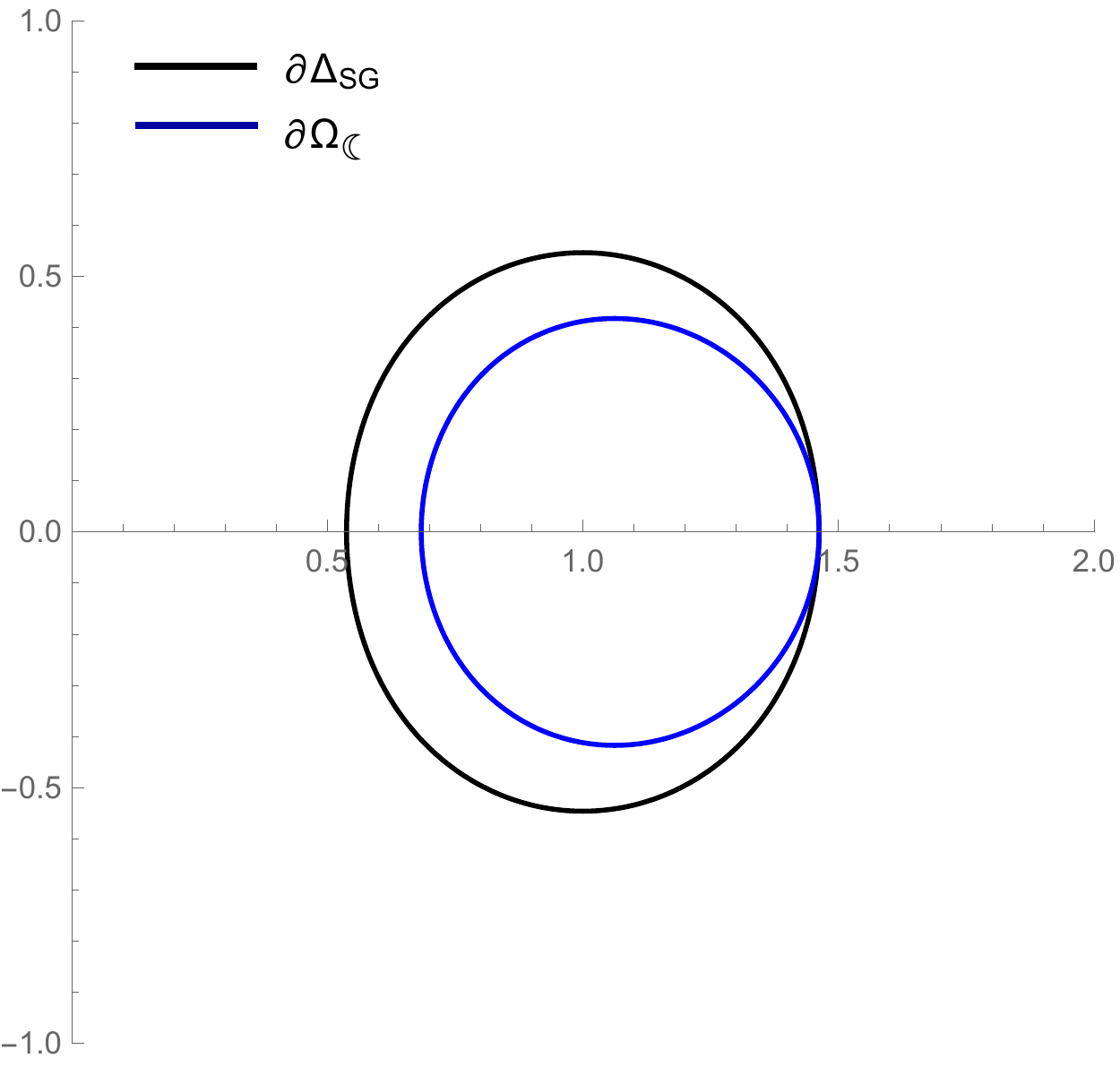}\\
          Sharpness of $ R_{\mathcal{S}^*_{SG}}(\mathcal{S}^*_{\leftmoon})$

           \end{tabular}

           & \begin{tabular}{l}
             \parbox{0.4\linewidth}{
            \includegraphics[height=5cm, width=5cm]{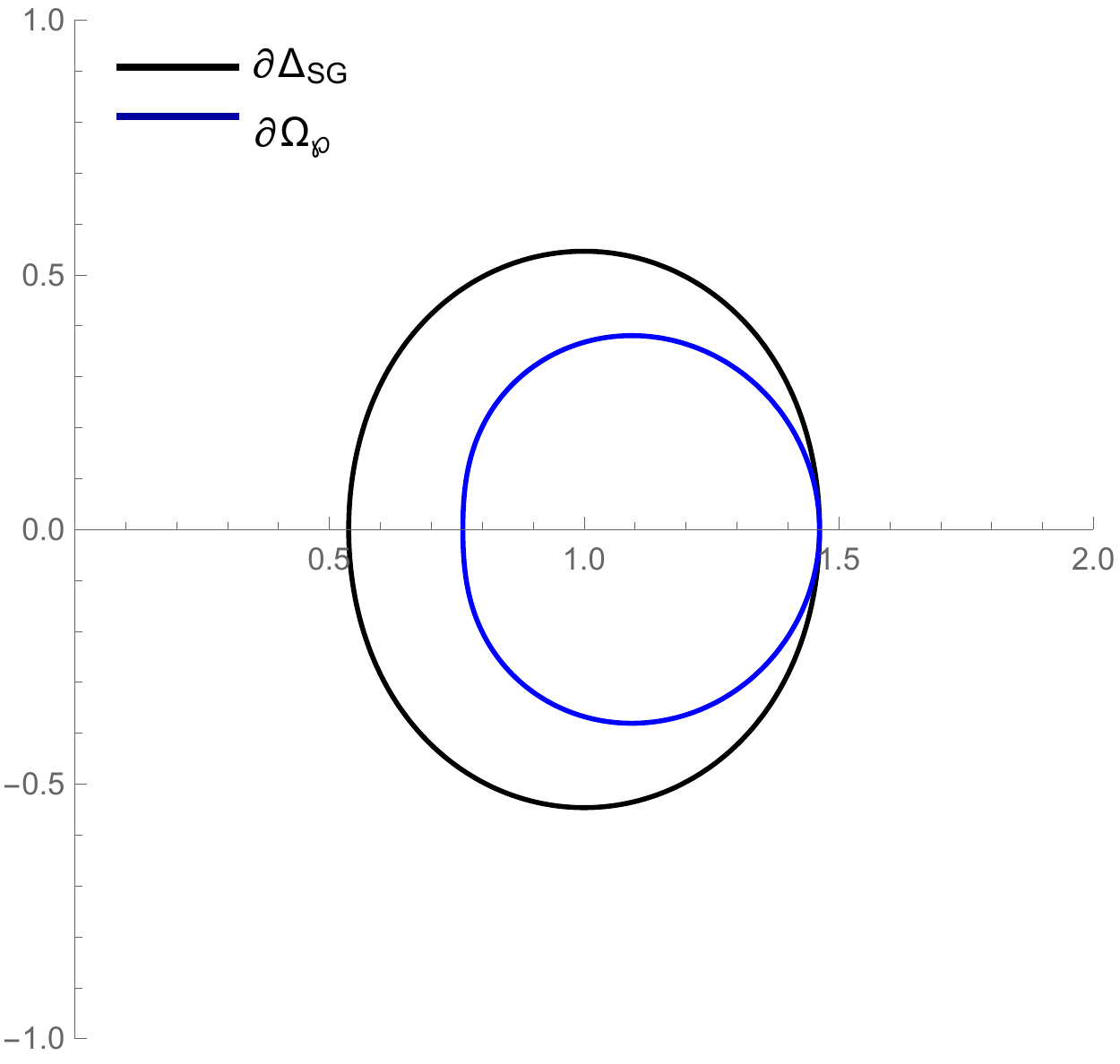}\\
              Sharpness of $ R_{\mathcal{S}^*_{SG}}(\mathcal{S}^*_{\wp})$
               }
         \end{tabular} \\
\end{tabular}
\caption{}
\end{figure}
\end{proof}
Now we consider the following classes for our next result. The class $\mathcal{S}^*_{Ne},$ defined by Wani and Swaminathan~\cite{lateef} by taking $\phi$ as $1+z-z^3/3$ and the class $\mathcal{S}^*_S=\mathcal{S}^*(1+\sin{z})$, introduced by Cho et al.~\cite{vktsine}.
\begin{theorem}
 The $\mathcal{S}^*_{SG}$-radius for the classes $\mathcal{S}^*_{Ne}$ and $\mathcal{S}^*_S$ is given by:
 \begin{itemize}
 \item[$(i)$] $R_{\mathcal{S}^*_{SG}}(\mathcal{S}^*_{Ne})=r_{Ne}\approx 0.43473,$ which is the smallest positive root of the equation $(e+1)(3r+r^3)=3(e-1).$
 \item[$(ii)$]  $R_{\mathcal{S}^*_{SG}}(\mathcal{S}^*_S)=\log\left(\frac{\sqrt{2 \left(1+e^2\right)}+e-1}{1+e}\right)\approx 0.447074.$
 \end{itemize}
\end{theorem}
\begin{proof}
\begin{itemize}
\item[$(i)$] Let $f\in\mathcal{S}^*_{Ne},$ then $zf'(z)/f(z)\prec 1+z-z^3/3.$ Thus on $|z|=r,$
$$\left|\dfrac{zf'(z)}{f(z)}-1\right|=|z-\dfrac{z^3}{3}|\leq r+\dfrac{r^3}{3},$$
which if, not greater than $(e-1)/(e+1),$ is sufficient to conclude that $f\in\mathcal{S}^*_{SG}$ by Lemma~\ref{main}. Note that for $r\leq r_{Ne},$ the above inequality holds and hence the result.
\item[$(ii)$] Let $f\in\mathcal{S}^*_S,$ then $zf'(z)/f(z)\prec 1+\sin{z}.$ So on $|z|=r,$ we have
  \begin{equation*}
    \left|\frac{zf'(z)}{f(z)}-1\right|=|\sin{z}|\leq \sinh{r}.
  \end{equation*}
  Now by using Lemma~\ref{main}, we observe that for $f$ to belong to $\mathcal{S}^*_{SG},$ it is sufficient to prove that $\sinh{r}\leq (e-1)/(e+1).$ Solving this inequality, we obtain the desired result.
\end{itemize}
\end{proof}
Further let us consider the classes, introduced by Ali et al.~\cite{ratio}, defined as follows:
\begin{eqnarray*}
  \mathcal{G}_1 &=& \left\{f\in\mathcal{A}_n:\frac{f}{g}\in\mathcal{P}_n\;\text{and}\;\frac{g(z)}{z}\in\mathcal{P}_n,\;g\in\mathcal{A}_n\right\}\\
 \mathcal{G}_2 &=& \left\{f\in\mathcal{A}_n:\frac{f}{g}\in\mathcal{P}_n\;\text{and}\;\frac{g(z)}{z}\in\mathcal{P}_n(1/2),\;g\in\mathcal{A}_n\right\}  \\
  \mathcal{G}_3 &=& \left\{f\in\mathcal{A}_n:\left|\frac{f(z)}{g(z)}-1\right|<1\;\text{and}\;\frac{g(z)}{z}\in\mathcal{P}_n,\;g\in\mathcal{A}_n\right\}  \\
 \mathcal{G}_4 &=& \left\{f\in\mathcal{A}_n:\left|\frac{f(z)}{g(z)}-1\right|<1,\;g\in\mathcal{C}_n\right\}.
\end{eqnarray*}
\begin{theorem}
The $\mathcal{S}^*_{SG,n}$ radius of the class $\mathcal{G}_1$ is given by
\begin{equation*}
  R_{\mathcal{S}^*_{SG,n}}(\mathcal{G}_1)=r_{\mathcal{G}_1}:=\frac{(e-1)^{\frac{1}{n}}}{(2n(1+e)+\sqrt{4n^2(1+e)^2+(e-1)^2})^{\frac{1}{n}}}.
\end{equation*}
The estimate is sharp.
\end{theorem}
\begin{proof}
  Let $f\in\mathcal{G}_1,$ then there exists  $p,\;q\in\mathcal{P}_n$ such that
  \begin{equation*}
    p(z)=\frac{g(z)}{z}\quad\text{and}\quad q(z)=\frac{f(z)}{g(z)}.
  \end{equation*}
 By Lemma~\ref{lemmab}, we have on $|z|=r,$ the following inequalities
  \begin{equation}\label{pqr}
    \left|\frac{zp'(z)}{p(z)}\right|\leq\frac{2n r^n}{1-r^{2n}}\quad\text{and}\quad\left|\frac{zq'(z)}{q(z)}\right|\leq\frac{2nr^n}{1-r^{2n}}.
  \end{equation}
Further, note that $f(z)=g(z)q(z)=zp(z)q(z),$ which upon logarithmic differentiation yields
  \begin{equation}\label{zfp}
    \frac{zf'(z)}{f(z)}=1+\frac{zp'(z)}{p(z)}+\frac{zq'(z)}{q(z)}.
  \end{equation}
  Using~\eqref{pqr}, we obtain
  \begin{equation*}
  \left|\frac{zf'(z)}{f(z)}-1\right|\leq \frac{4nr^n}{1-r^{2n}}.
  \end{equation*}
  By Lemma~\ref{main}, it is clear that $f\in\mathcal{S}^*_{SG}$ if the quantity $4nr^n/(1-r^{2n})$ does not exceed $(e-1)/(e+1).$ This  leads us to $(e-1)r^{2n}+4n(e+1)r^n-(e-1)\leq 0.$ Solving this inequality, we obtain $r\leq r_{\mathcal{G}_1}.$ The sharpness of the bound can be verified by the functions
  \begin{equation*}
    f_1(z)=z\left(\frac{1+z^n}{1-z^n}\right)^2\quad\text{and}\quad g_1(z)=z\left(\frac{1+z^n}{1-z^n}\right).
  \end{equation*}
  Clearly $f_1,g_1\in\mathcal{A}_n$ and $f_1(z)/g_1(z)=g_1(z)/z,$ which belongs to $\mathcal{P}_n.$ So $f_1\in\mathcal{G}_1$ and
  \begin{equation*}
    \frac{zf_1'(z)}{f_1(z)}=1+\frac{4nz^n}{1-z^{2n}}.
  \end{equation*}
  The image domain of $\mathbb{D}$ mapped by $zf_1'(z)/f_1(z)$ touches the boundary of $\Delta_{SG}$ at the points $\pm r_{\mathcal{G}_1}.$
\end{proof}
\begin{theorem}
The $\mathcal{S}^*_{SG,n}$ radius of the class $\mathcal{G}_2$ is given by
\begin{equation*}
  R_{\mathcal{S}^*_{SG,n}}(\mathcal{G}_2)=r_{\mathcal{G}_2}:=\frac{(2(e-1))^{\frac{1}{n}}}{(3n(1+e)+\sqrt{(3n(e+1))^2+4(e-1)(n(e+1)+(e-1))})^{\frac{1}{n}}}.
\end{equation*}
The estimate is sharp.
\end{theorem}
\begin{proof}
Let $f\in\mathcal{G}_2.$ Now, let us define
  \begin{equation*}
    p(z)=\frac{g(z)}{z}\quad\text{and}\quad q(z)=\frac{f(z)}{g(z)}
  \end{equation*}
  so that $f(z)=zp(z)q(z),$ where $p\in\mathcal{P}_n(1/2)$ and $q\in\mathcal{P}_n.$ From~\eqref{zfp} and Lemma~\ref{lemmab}, we have
    \begin{equation*}
    \left|\frac{zf'(z)}{f(z)}-1\right|\leq\frac{2n r^n}{1-r^{2n}}+\frac{nr^n}{1-r^n}=\frac{nr^{2n}+3nr^n}{1-r^{2n}}\quad\text{on}\;|z|=r.
  \end{equation*}
  In view of Lemma~\ref{main}, $f$ is a member of $\mathcal{S}^*_{SG}$ if $(nr^{2n}+3nr^n)/(1-r^{2n})\leq (e-1)/(e+1).$ Equivalently, $(n(e+1)+e-1)r^{2n}+3n(e+1)r^n-(e-1)\leq 0.$ Solving this inequality, we obtain $r\leq r_{\mathcal{G}_2}.$ The sharpness of the bound can be verified by the functions
  \begin{equation*}
    f_2(z)=\frac{z(1+z^n)}{(1-z^n)^2}\quad\text{and}\quad g_2(z)=\frac{z}{1-z^n}.
  \end{equation*}
  Clearly $f_2,g_2\in\mathcal{A}_n,$ $f_2/g_2\in\mathcal{P}_n$ and $g_2/z\in\mathcal{P}(1/2).$ Hence $f_2$ is a member of $\mathcal{G}_2$ and
  \begin{equation*}
    \frac{zf_2'(z)}{f_2(z)}=\frac{1+3nz^n+(n-1)z^{2n}}{1-z^{2n}}.
  \end{equation*}
The radius estimate is sharp since $zf'_2(z)/f_2(z)$ maps $\mathbb{D}$ onto a domain which touches the boundary of $\Delta_{SG}$ at $z=r_{\mathcal{G}_2}.$
\end{proof}
\begin{theorem}
The $\mathcal{S}^*_{SG,n}$ radius of the class $\mathcal{G}_3$ is given by
\begin{equation*}
  R_{\mathcal{S}^*_{SG,n}}(\mathcal{G}_3)=r_{\mathcal{G}_3}:=\frac{(2(e-1))^{\frac{1}{n}}}{(3n(1+e)+\sqrt{(3n(e+1))^2+4(e-1)(n(e+1)+(e-1))})^{\frac{1}{n}}}.
\end{equation*}
The estimate is sharp.
\end{theorem}
\begin{proof}
Let $f\in\mathcal{G}_3.$ So, we may define
  \begin{equation*}
    p(z)=\frac{g(z)}{z}\quad\text{and}\quad q(z)=\frac{g(z)}{f(z)}.
  \end{equation*}
In view of the following implication
  \begin{equation*}
    \left|\dfrac{f(z)}{g(z)}-1\right|<1\iff \frac{g}{f}\in\mathcal{P}_n\left(\frac{1}{2}\right),
  \end{equation*}
it is obvious that $p\in\mathcal{P}_n$ and $q\in\mathcal{P}_n(1/2).$ Note that $f(z)=zp(z)/q(z)$ and thus
\begin{equation*}
  \frac{zf'(z)}{f(z)}=1+\frac{p'(z)}{p(z)}-\frac{zq'(z)}{q(z)}.
\end{equation*}
By using Lemma~\ref{lemmab}, we obtain
    \begin{equation*}
    \left|\frac{zf'(z)}{f(z)}-1\right|\leq\frac{nr^{2n}+3nr^n}{1-r^{2n}}\quad\text{on}\;|z|=r.
  \end{equation*}
  In order to prove $f\in\mathcal{S}^*_{SG},$  it suffices to show that $(nr^{2n}+3nr^n)/(1-r^{2n})\leq (e-1)/(e+1),$ in view of Lemma~\ref{main}. This inequality reduces to $(n(e+1)+e-1)r^{2n}+3n(e+1)r^n-(e-1)\leq 0.$ Solving this, we obtain $r\leq r_{\mathcal{G}_3}.$ The sharpness of the bound can be verified by the functions
  \begin{equation*}
    f_3(z)=\frac{z(1+z^n)^2}{1-z^n}\quad\text{and}\quad g_3(z)=\frac{z(1+z^n)}{1-z^n}.
  \end{equation*}
  Note that
  \begin{equation*}
    \left|\frac{f_3(z)}{g_3(z)}-1\right|=|z|^n <1\quad\text{and}\quad\frac{g_3(z)}{z}=\frac{1+z^n}{1-z^n}\in\mathcal{P}_n,
 \end{equation*}
  which implies that $f_3\in\mathcal{G}_3.$ Moreover, it can be observed that the domain $(zf'_3/f_3)(\mathbb{D})$ touches the boundary of $\Delta_{SG}$ at $-r_{\mathcal{G}_3}.$
\end{proof}
\begin{theorem}
The $\mathcal{S}^*_{SG,n}$ radius of the class $\mathcal{G}_4$ is given by
\begin{equation*}
  R_{\mathcal{S}^*_{SG,n}}(\mathcal{G}_4)=r_{\mathcal{G}_4}:=\frac{(2(e-1))^{\frac{1}{n}}}{((n+1)(1+e)+\sqrt{(n+1)^2(1+e)^2+4(e-1)((e+1)n-2)})^{\frac{1}{n}}}.
\end{equation*}
The estimate is sharp.
\end{theorem}
\begin{proof}
Let $f\in\mathcal{G}_4.$ Now, suppose $q(z)=g(z)/f(z),$ where $g\in\mathcal{A}_n$ is a convex function.
As deduced in the last theorem, we know that $q\in\mathcal{P}_n(1/2).$ By using Lemma~\ref{lemmab}, we obtain
    \begin{equation}\label{q}
    \left|\frac{zq'(z)}{q(z)}\right|\leq\frac{nr^n}{1-r^{n}},\quad |z|=r.
  \end{equation}
  Also since $g$ is convex, we have $zg'/g$ is in $\mathcal{P}(1/2).$ Hence
  \begin{equation}\label{g}
    \left|\frac{zg'(z)}{g(z)}-\frac{1}{1-r^{2n}}\right|\leq \frac{r^n}{1-r^{2n}}
  \end{equation}
  by Lemma~\ref{lemmab}. Since $f=g/q,$ we obtain
  \begin{equation*}
    \frac{zf'(z)}{f(z)}=\frac{zg'(z)}{g(z)}-\frac{zq'(z)}{q(z)}.
  \end{equation*}
  In view of~\eqref{q} and~\eqref{g}, we have
\begin{equation*}
  \left|\frac{zf'(z)}{f(z)}-\frac{1}{1-r^{2n}}\right|\leq \frac{(n+1)r^n+nr^{2n}}{1-r^{2n}}.
\end{equation*}
 For each $r\in(0,1),$ the quantity $1-r^{2n}$ is less than 1 and thus Lemma~\ref{main} implies that $f\in\mathcal{S}^*_{SG}$ provided
 \begin{equation*}
   \frac{(n+1)r^n+nr^{2n}}{1-r^{2n}}\leq\frac{1}{1-r^{2n}}-\frac{2}{1+e}.
 \end{equation*}
 This inequality reduces to
 \begin{equation*}
   ((1+e)n-2)r^{2n}+(n+1)(1+e)r^n+1-e\leq 0.
 \end{equation*}
  Solving this inequality, we obtain $r\leq r_{\mathcal{G}_4}.$ The functions
  \begin{equation*}
    f_4(z)=\frac{z(1+z^n)}{(1-z^n)^{1/n}}\quad\text{and}\quad g_4(z)=\frac{z}{(1-z^n)^{1/n}}
  \end{equation*}
  validate the sharpness of the bound. Note that $|f_4(z)/g_4(z)-1|=|z|^n <1$ and $g_4\in\mathcal{C}_n,$ which implies $f_4\in\mathcal{G}_4.$  Moreover, it can be observed that $zf'_4/f_4$ maps $\mathbb{D}$ onto a domain that touches the boundary of $\Delta_{SG}$ at $\pm r_{\mathcal{G}_4}.$
\end{proof}
For $0\leq\alpha<1,$ Reade~\cite{reade} introduced the class of close-to-starlike functions of type $\alpha,$ given by
\begin{equation*}
  \mathcal{CS}^*(\alpha):=\left\{f\in\mathcal{A}_n:\frac{f}{g}\in\mathcal{P}_n,\;g\in\mathcal{S}^*(\alpha)\right\}.
\end{equation*}
\begin{theorem}\label{cstheo}
  The sharp $\mathcal{S}^*_{SG,n}$ radius for the class $\mathcal{CS}^*_n(\alpha)$ is given by
\begin{eqnarray*}
  R_{\mathcal{S}^*_{SG,n}}(\mathcal{CS}^*_n(\alpha))=r_{cs}&:=&(e-1)/((1+e)(1+n-\alpha)\\
  &+&\sqrt{(e+1)^2(1+n-\alpha)^2+(e-1)((1-2\alpha)(e+1)+2e)})
\end{eqnarray*}
\end{theorem}
\begin{proof}
  Let $f\in\mathcal{CS}^*_n(\alpha).$ Then there is some function $g\in\mathcal{S}^*_n(\alpha)$ such that $h=f/g\in\mathcal{P}_n.$ Using Lemma~\ref{lemmab}, we get
  \begin{equation}\label{h}
    \left|\frac{zh'(z)}{h(z)}\right|\leq \frac{2nr^n}{1-r^{2n}}.
  \end{equation}
Since $g\in\mathcal{S}^*_n(\alpha)$ and $zg'/g\in\mathcal{P}_n(\alpha),$ again by using Lemma~\ref{lemmab}, we obtain
\begin{equation}\label{gr}
  \left|\frac{zg'(z)}{g(z)}-\frac{1+(1-2\alpha)r^{2n}}{1-r^{2n}}\right|\leq \frac{2(1-\alpha)r^n}{1-r^{2n}}.
\end{equation}
Since $f=hg,$ we have
\begin{equation*}
  \frac{zf'(z)}{f(z)}=\frac{zg'(z)}{g(z)}+\frac{zh'(z)}{h(z)}
\end{equation*}
From~\eqref{h} and~\eqref{gr}, it follows that
\begin{equation*}
  \left|\frac{zf'(z)}{f(z)}-\frac{1+(1-2\alpha)r^{2n}}{1-r^{2n}}\right|\leq \frac{2(1+n-\alpha)r^n}{1-r^{2n}}.
\end{equation*}
Note that the above inequality represents a disk with center $a$ and radius $r_0$ given by
\begin{equation*}
a:=\frac{1+(1-2\alpha)r^{2n}}{1-r^{2n}}\quad\text{and}\quad r_0:=\frac{2(1+n-\alpha)r^n}{1-r^{2n}}.
\end{equation*}
Since $a>1,$ we have from Lemma~\ref{main} that $f$ is a member of the class $\mathcal{S}^*_{SG}$ if $r_0\leq (2e)/(e+1)-a$ or equivalently,
\begin{equation*}
  ((1-2\alpha)(e+1)+2e)r^{2n}+2(e+1)(1+n-\alpha)r^n+1-e\leq 0.
\end{equation*}
Solving this inequality, we obtain $r\leq r_{cs}.$ The sharpness of this result can be verified by the following functions.
\begin{equation*}
  f(z)=\frac{z(1+z^n)}{(1-z)^{(n+2-2\alpha)/n}}\quad\text{and}\quad g(z)=\frac{z}{(1-z^n)^{(2-2\alpha)/n}}.
\end{equation*}
Note that $f/g=(1+z^n)/(1-z^n)\in\mathcal{P}_n$ and $g\in\mathcal{S}^*_n(\alpha),$ which ensures that $f\in\mathcal{CS}^*_n(\alpha).$
\end{proof}
For our next result, let us recall the class $\mathcal{W}_n,$ given as
\begin{equation*}
  \mathcal{W}_n:=\{f\in\mathcal{A}_n: f/z\in\mathcal{P}_n\},
\end{equation*}
which was introduced by MacGregor~\cite{mac}.
\begin{theorem}
  The sharp $\mathcal{S}^*_{SG,n}$ radius for $\mathcal{W}_n$ is
  \begin{equation*}
    R_{\mathcal{S}^*_{SG,n}}(\mathcal{W}_n):=r_w=\left(\frac{e-1}{\sqrt{n^2(e+1)^2+(e-1)^2}+n(e+1)}\right)^{1/n}.
  \end{equation*}
\end{theorem}
\begin{proof}
  Let $f\in\mathcal{W}_n,$ then there exists some $h\in\mathcal{P}_n$ such that $h(z)=f(z)/z.$ Using Lemma~\ref{lemmab}, we have
  \begin{equation*}
    \left|\frac{zh'(z)}{h(z)}\right|\leq \frac{2nr^n}{1-r^{2n}}.
  \end{equation*}
  This further implies that
  \begin{equation*}
    \left|\frac{zf'(z)}{f(z)}-1\right|=\left|\frac{zh'(z)}{h(z)}\right|\leq \frac{2nr^n}{1-r^{2n}}.
  \end{equation*}
  Now, by using~\ref{main} it follows that $f\in\mathcal{S}^*_{SG}$ provided $2nr^n/(1-r^{2n})\leq (e-1)/(e+1).$ This inequality, after a few steps, reduces to the following
  \begin{equation*}
    (e-1)r^{2n}+2n(e+1)r^n-(e-1)\leq 0.
  \end{equation*}
  Solving this inequality gives us $r\leq r_w.$ The sharpness of the result can be verified by the function $f_w(z)=z(1+z^n)/(1-z^n).$ Clearly this function satisfy $f(z)/z\in\mathcal{P}_n$ and thus $f_w\in\mathcal{W}_n.$
  Also, at the points $\pm r_w$ the function $zf'_w/f_w$ touches the boundary of $\Delta_{SG}.$
\end{proof}
For $\beta>1,$ the class $\mathcal{M}(\beta)$ introduced by Uralegaddi et al.~\cite{urlagaddi} is given by
\begin{equation*}
  \mathcal{M}(\beta):=\left\{f\in\mathcal{A}_n:\RE\frac{zf'(z)}{f(z)}<\beta,\;z\in\mathbb{D}\right\}.
\end{equation*}
In terms of subordination, the above class can be written as
\begin{equation*}
  \mathcal{M}(\beta):=\left\{f\in\mathcal{A}_n:\frac{zf'(z)}{f(z)}\prec \frac{1+(1-2\beta)z^n}{1-z^n},\;z\in\mathbb{D}\right\}.
\end{equation*}
\begin{theorem}
  The $\mathcal{S}^*_{SG,n}$ radius for the class $\mathcal{M}(\beta)$ is given by
\begin{equation*}
\mathcal{S}^*_{SG}(\mathcal{M}_n(\beta))=\left(\frac{e-1}{(e-1)+(e+1)(\beta)-1}\right)^{1/n}.
\end{equation*}
The result is sharp for the function $f(z)=z(1-z^n)^{(2(\beta-1)/n)}.$
\end{theorem}
The proof of the above theorem is omitted here as it is much similar to the proof of Theorem~\ref{cstheo}.
Next we consider the class $\mathcal{C}(\alpha)\;(0\leq\alpha<1),$ of convex functions of order $\alpha.$ Note that this class is a generalization of the class $\mathcal{C},$ which can be obtained by setting $\alpha=0.$
\begin{theorem}
Let $f\in\mathcal{S}^*_{SG},$ then f is convex of order $\alpha$ in $|z|<r_{\mathcal{C}}(\alpha),$ where $r_{\mathcal{C}}(\alpha)\in(0,1)$ is the smallest positive root of
\begin{equation*}
e^r(r+\alpha)-2+\alpha=0.
\end{equation*}
\end{theorem}
\begin{proof}
Let $f\in\mathcal{S}^*_{SG},$ then $zf'(z)/f(z)\prec 2/(1+e^{-z})$ and thus there exists a function $\omega$ with $\omega(0)=0$ and $|\omega(z)|<1$ such that
\begin{equation*}
\frac{zf'(z)}{f(z)}=\frac{2}{1+e^{-\omega(z)}}.
\end{equation*}
Differentiating the above equation logarithmically, we obtain
\begin{equation*}
1+\frac{zf''(z)}{f'(z)}=\frac{2}{1+e^{-\omega(z)}}+\frac{\omega'(z)e^{-\omega(z)}}{1+e^{-\omega(z)}}.
\end{equation*}
Note that
\begin{eqnarray*}
\RE\left(1+\dfrac{zf''(z)}{f'(z)}\right) &= &\RE\left(\dfrac{2}{1+e^{-\omega(z)}}+\dfrac{\omega'(z)e^{-\omega(z)}}{1+e^{-\omega(z)}}\right)\\
& \geq & \RE\left(\dfrac{2}{1+e^{-\omega(z)}}\right)-\left|\dfrac{\omega'(z)e^{-\omega(z)}}{1+e^{-\omega(z)}}\right|\\
& \geq & \dfrac{2}{1+e^r}-\dfrac{re^r}{1+e^r},
\end{eqnarray*}
which is greater than $\alpha$ provided $r\leq r_{\mathcal{C}}(\alpha).$ The result is sharp for the function
\begin{equation*}
	f_0(z)=z\exp{\left(\int_0^z \dfrac{e^t-1}{t(e^t+1)}dt\right)}=z + \dfrac{z^2}{2} + \dfrac{z^3}{8} + \dfrac{z^4}{144} - \dfrac{5 z^5}{1152}+\cdots,
\end{equation*}
\end{proof}
\begin{corollary}
The sharp radius of convexity for the functions in $\mathcal{S}^*_{SG}$ is $r\approx 0.852606.$
\end{corollary}
%

\end{document}